\documentclass[english,12pt]{article}
\usepackage{amssymb}

\usepackage{bm}
\usepackage{amsmath}
\usepackage{amsfonts}
\usepackage{mathrsfs}
\usepackage{graphicx}
\usepackage{tikz}
\usetikzlibrary{arrows,shapes,positioning}
\usetikzlibrary{decorations.markings}
\tikzstyle arrowstyle=[scale=1.5]
\tikzstyle directed=[postaction={decorate,decoration={markings, mark=at position .9 with {\arrow[arrowstyle]{stealth}}}}]
\tikzstyle reverse directed=[postaction={decorate,decoration={markings, mark=at position .9 with {\arrowreversed[arrowstyle]{stealth};}}}]
\usepackage{}

\newtheorem{Theorem}{Theorem}[section]
\newtheorem{Corollary}{Corollary}[section]

\newtheorem{Conjecture}{Conjecture}[section]

\newtheorem{Lemma}{Lemma}[section]

\newtheorem{Example}{Example}[section]

\def\emptyset{\mbox{{\rm \O}}}

\newenvironment{proof}{
\noindent {\bf Proof.}\rm}%
{\mbox{}\hfill\rule{0.5em}{0.809em}\par}
\usepackage{calc}
\usepackage{tikz}
\usetikzlibrary{decorations.markings}
\tikzstyle{vertex}=[circle, draw, inner sep=0pt, minimum size=6pt]

\begin{document}
\title{\LARGE{\textbf{A new condition on dominated pair degree sum for a digraph to be supereulerian\thanks{This work was supported by the
Natural Science Foundation of Xinjiang (No. 2020D04046) and the
National Natural Science
Foundation of China (No. 12261016, No. 11901498, No. 12261085).}}}}
\author{{Changchang Dong$^{a,b}$, Jixiang Meng$^{b}$\footnote{Corresponding
author.
E-mail:  mjxxju@sina.com}, Juan Liu$^{c}$}
\\
\small  a. Department of Mathematical Sciences, Tsinghua University, Beijing 100084, China  \\
\small  b. College of Mathematics and System Sciences, Xinjiang University, Urumqi 830046, China  \\
\small  c. College of Big Data Statistics, Guizhou University of Finance and Economics, Guiyang 550025, China  \\}
\date{}
\maketitle
\noindent {\small {\bf Abstract}\\
A digraph $D$ is supereulerian if $D$ contains a spanning eulerian subdigraph.
 For any two vertices $u,v$ in a digraph $D$,
if $(u,w),(v,w)\in A(D)$ for some $w\in V(D)$, then we call the pair $\{u, v\}$ dominating;
if $(w,u),(w,v)\in A(D)$ for some $w\in V(D)$, then we call the pair $\{u, v\}$ dominated.
In 2015, Bang-Jensen and Maddaloni [Journal of graph theory, 79(1) (2015) 8-20]
 proved that if a strong digraph $D$ with $n$ vertices satisfies
 $d(u) + d(v)\geq 2n -3$ for any pair of nonadjacent vertices $\{u,v\}$ of $D$, then $D$ is supereulerian.
In this paper, we study the above degree sum condition for any pair of dominated or dominating nonadjacent vertices
of supereulerian digraphs.
}
\\



\section{Introduction}

Digraphs
considered are loopless and without parallel arcs.
We refer the reader to
\cite{BaGu09} for digraphs for undefined terms and notation.
In this paper, we
define $[k]=\{1,2,\ldots ,k\}$ for an integer $k>0$ and
 use $(w,z)$ to denote an arc oriented from a vertex
$w$ to a vertex $z$ and say that $w$ dominates $z$.
For any two vertices $u,v$ in a digraph $D$,
if $(u,w),(v,w)\in A(D)$ for some $w\in V(D)$, then we call that
$\{u, v\}$ dominates $w$ or call the pair $\{u, v\}$ \emph{dominating}; if $(w,u),(w,v)\in A(D)$ for some $w\in V(D)$, then we call that
$\{u, v\}$ is dominated by $w$ or call the pair $\{u, v\}$ \emph{dominated}.
We often write \emph{dipaths} for directed paths, \emph{dicycles} for directed cycles and \emph{ditrails} for directed trails in digraphs.
The \emph{length} of a ditrail is the number of its arcs.
If a ditrail $T$ starts
at $w$ and ends at $z$, we may call it $(w,z)$-ditrail $T$ or $T_{[w,z]}$ and say
$w$
is the \emph{initial vertex} of $T$ and $z$ is the \emph{terminal vertex} of $T$.
A $(w,z)$-ditrail of minimum length in $D$ is called a shortest $(w,z)$-ditrail in $D$.
A digraph $D$ is \emph{strong} if any vertex of a digraph $D$ is reachable from all other vertex of $D$.
We often write \emph{$|D|$} for $|V(D)|$
and use $K_n^*$ to represent the \emph{complete digraph} with $n$ vertices.
For a graph $G$, a digraph $D$ is called a biorientation of $G$ if $D$ is obtained from $G$ by replacing each edge $xy$ of $G$ by either $(x,y)$ or $(y,x)$ or the pair $(x,y)$ and $(y,x)$. Recall that a
\emph{semicomplete multipartite digraph}
is a biorientation of a complete multipartite graph.

Let $T=v_{1} v_{2} \cdots v_{k}$ denote a ditrail.
For any $1 \leq i \leq j \leq k$, we use $T_{[v_{i}, v_{j}]}=v_{i} v_{i+1} \cdots v_{j-1} v_{j}$ to denote the \emph{sub-ditrail} of $T$. Likewise, if $Q=u_{1} u_{2} \cdots u_{k} u_{1}$ is a closed ditrail, then for any $i, j$ with $1 \leq i<j \leq k, Q_{[u_{i}, u_{j}]}$ denotes the sub-ditrail $u_{i} u_{i+1} \cdots u_{j-1} u_{j}$. If $T^{\prime}=w_{1} w_{2} \cdots w_{k^{\prime}}$ is a ditrail with $v_{k}=w_{1}$ and $V(T) \cap V\left(T^{\prime}\right)=\left\{v_{k}\right\}$, then we use $T T^{\prime}$ or $T_{[v_{1}, v_{k}]} T^{\prime}_{[v_{k}, w_{k^{\prime}}]}$ to denote the ditrail $v_{1} v_{2} \cdots v_{k} w_{2} \cdots w_{k^{\prime}}$. If $V(T) \cap V\left(T^{\prime}\right)=\emptyset$ and there is a dipath $z_{1} z_{2} \cdots z_{t}$ with
$z_{2}, \cdots, z_{t-1} \notin V(T) \cup V\left(T^{\prime}\right)$ and with $z_{1}=v_{k}$ and $z_{t}=w_{1}$, then we use $Tz_{1} \cdots z_{t} T^{\prime}$ to denote the ditrail $v_{1} v_{2} \cdots v_{k} z_{2} \cdots z_{t} w_{2} \cdots w_{k^{\prime}}$. In particular, if $T$ is a $(v, w)$-ditrail of a digraph $D$ and $(u ,v), (w, z) \in A(D)-A(T)$, then we use $u v T w z$ to denote the $(u, z)$-ditrail $D\langle A(T) \cup\{(u ,v), (w, z)\}\rangle$. The subdigraphs $u v T$ and $T w z$ are similarly defined.

For a digraph $D$, $a\in A(D)$ and a subdigraph $S$ of $D$,
we use $D-S$ to denote the subdigraph $D\langle V(D)-V(S)\rangle$,
use $D-a$ to denote the subdigraph $D\langle A(D)-a\rangle$,
and use $D+a$ to denote the subdigraph $D\langle A(D)+a\rangle$.
Let $D_{1}$ and $D_{2}$ be two digraphs, the \emph{union} $D_{1} \cup D_{2}$ of $D_{1}$ and $D_{2}$ is a digraph with vertex set $V\left(D_{1} \cup D_{2}\right)=V\left(D_{1}\right) \cup V\left(D_{2}\right)$ and arc set $A\left(D_{1} \cup D_{2}\right)=A\left(D_{1}\right) \cup A\left(D_{2}\right)$.
For $S,T\subseteq V(D)$, an ($s,t$)-dipath $P$ is an $(S,T)$\emph{-dipath} if $s\in S$, $t\in T$ and $V(P)\cap (S\cup T)=\{s,t\}$.
Note that if $S\cap T\not =\emptyset$, then a vertex $s\in S\cap T$ forms an ($S,T$)-dipath by itself.
When $S$ and $T$ are subdigraphs of $D$, we also talk about an ($S,T$)-dipath.

Let $d_D^-(s),d_D^+(s),d_D(s) = d_D^-(s)+d_D^+(s), N_D^-(s)$ and $N_D^+(s)$ denote, respectively, the \emph{in-degree, out-degree, degree, in-neighbourhood} and
 \emph{out-neighbourhood} of a vertex $s\in V(D)$.

In \cite{BoST77}, Boesch et al. raised the
supereulerian problem, which strives to describe graphs
that contain spanning eulerian subgraphs.
In \cite{Pull79}, Pulleyblank
showed that determining whether a graph is supereulerian, even within planar graphs,
is NP-complete. There
have been many studies on this topic, as revealed in
the surveys  \cite{Catl92, ChenLai95} and \cite{LaSY13}.

It is natural
to try to relate supereulerian graphs to supereulerian digraphs.
 A digraph $D$ is \emph{supereulerian} if it contains a closed ditrail $S$ with $V(S) = V(D)$, i.e.,
it has a spanning eulerian subdigraph,
and \emph{nonsupereulerian} otherwise.
Results on supereulerian digraphs can be found in
\cite{AlLa15, Al22, BaMa14, Dong21, HLL14, HLL16}, among others.
It is worth pointing out that only a few of degree sum conditions are
studied to ensure supereulerianicity in digraphs.
In particular, the following have been proved.

\begin{Theorem}\label{BaMa14 2n-3} (\cite{BaMa14})
Let $D$ be a strong digraph with $n$ vertices.
If for any pair of nonadjacent
vertices $u$ and $v$,
$d(u)+d(v)\geq 2n-3$, then $D$ is supereulerian.
\end{Theorem}

It is
observed that in Theorem \ref{BaMa14 2n-3}, degree sum conditions on every pairs of nonadjacent
vertices are needed to warrant the digraph to be supereulerian.
The main purpose of this paper is to
 consider a degree sum condition about pairs of dominated (dominating) nonadjacent vertices
no longer on all pairs of nonadjacent vertices.
First, we give the following conjectures.
If they are true, then each of them can be seen as a generalization of Theorem \ref{BaMa14 2n-3}.
\begin{Conjecture}\label{di 2n-3 1}
Let $D$ be a strong digraph with $n$ vertices.
If for any pair of dominated nonadjacent vertices $\{u,v\}$ of $D$,
 $d(u)+ d(v)\geq 2n -3$, then $D$ is supereulerian.
\end{Conjecture}

\begin{Conjecture}\label{di 2n-3 2}
Let $D$ be a strong digraph with $n$ vertices.
If for any pair of dominated or dominating nonadjacent vertices $\{u,v\}$ of $D$,
$d(u)+ d(v)\geq 2n -3$, then $D$ is supereulerian.
\end{Conjecture}

These conjectures seem quite difficult to prove, but we are able to prove them in the
special cases.
In this paper, we prove that the Conjecture \ref{di 2n-3 1}
holds for semicomplete multipartite digraphs
and provide some supports for the Conjecture \ref{di 2n-3 2}
by showing that if a strong digraph $D$ with $n$ vertices satisfies
$d(u)+ d(v)\geq 2n -3$ and min$\{d^-(u) + d^+(v), d^+(u) + d^-(v)\} \geq n -2$, or satisfies
 $d(u) + d(v)\geq \frac{5}{2}n -\frac{11}{2}$, for any pair of dominated or dominating nonadjacent vertices $\{u,v\}$ of $D$, then $D$ is supereulerian.
 All our results are sharp.

\section{Main results}

We need the following lemmas and corollary.

\begin{Lemma}\label{notST}(\cite{Al22})
Let $D$ be a digraph, $S=u_1u_2\cdots u_s$ and $T=v_1v_2\cdots v_t$ be two arc distinct ditrails of $D$.
If $D$ does not contain a $(u_1,u_s)$-ditrail with vertex set $V(S)\cup V(T)$, then $d_S^-(v_1) + d_S^+(v_t)\leq |S|$.
\end{Lemma}

\begin{Corollary}\label{notSx}
Let $D$ be a digraph, $S=u_1u_2\cdots u_s$ be a ditrail in $D$ and $x\in V(D)$.
If $D$ does not contain a $(u_1,u_s)$-ditrail with vertex set $V(S)\cup \{x\}$, then $d_S(x) \leq |S|$.
\end{Corollary}

\begin{Lemma}\label{SMD} (\cite{BaMa14})
Let $u,v, w$ be vertices of a semicomplete multipartite digraph $D$, such
that there is an arc between $u, v$. Then there is an arc between $w$ and $\{u,v\}$ in $D$.
\end{Lemma}

\begin{Theorem}\label{smd 2n-3}
Let $D$
be a strong semicomplete multipartite digraph with $n$ vertices.
If $d(u)+ d(v)\geq 2n -3$ for any pair of dominated nonadjacent vertices $\{u,v\}$ of $D$, then $D$ is supereulerian.
\end{Theorem}
\begin{proof}
Suppose to the contrary
that $D$ is a nonsupereulerian semicomplete multipartite digraph.
Let $S$ be a closed ditrail with $|S|=s$ maximized in $D$.
Then $s< n$.
Since $D$ is strong, there exists an $(S, S)$-dipath $T$ with $|T|\geq 3$.
 Choose an $(S, S)$-dipath $T$
such that the length of the ditrail $P$ is minimum in $S$, where $P$ is a shortest $(u,v)$-ditrail
which
travels along $S$ from $u$ to $v$ such that
the initial vertex $u$ of $P$ is the initial vertex of $T$ and the terminal vertex $v$ of $P$ is the terminal vertex of $T$.
W.l.o.g., write $T=y_0x_1x_2\cdots x_ty_{p+1}$, that is $u=y_0$ and $v=y_{p+1}$.
Let $W = \{y_1,y_2,\cdots, y_p\}$ be the set of internal vertices of $P$,
$P_1$ be a longest ditrail which
travels along $S$ from $y_{p+1}$ to $y_0$ and $R=D-S$.
Then $|P_1|=s-p+c$, where $c= |W\cap P_1|.$
By the maximality of $S$, we have $y_0\not=y_{p+1},$ $(y_0,y_{p+1})\notin A(S)$ and $p\geq 1$.

By the choice of $T$ and the maximality of $S$, for any $i\in [t]$, we have $d_W(x_i)=0$.
If $d_W^+(x_i)>0$ or $d_W^-(x_i)>0$, then w.l.o.g., we may assume that $d_W^+(x_i)>0$, and so there exists a vertex $y_j\in W$
($j\in [p]$) such that $(x_i,y_j)\in A(D)$.
If $2\leq j\leq p$, then we can get another
$(S,S)$-dipath $T'$ with the initial vertex $y_0$ and the terminal vertex
$y_j$ such that the length of $(y_0,y_j)$-ditrail $P'$ in $S$ is less then
the length of $P$ in $S$, contrary to the choice of $T$ above. If $j=1$, then $S\cup T_{[y_0,x_i]}+(x_i,y_1)-(y_0,y_1)$ is a closed ditrail
with $|S\cup T_{[y_0,x_i]}+(x_i,y_1)-(y_0,y_1)|>|S|$, contrary to the maximality of $S$.
Therefore $d_W^+(x_i)=0$.
The proof for $d_W^-(x_i)=0$ is similar.
In particular, $x_i$ and $y_j$ are nonadjacent, for $i\in [t]$ and $j\in [p]$.

If $t\geq2$, then there exist two vertices $x_1,x_2$ with $(x_1,x_2)\in A(D)$.
By Lemma \ref{SMD}, there is an arc between $y_1$ and $\{x_1,x_2\}$, a contradiction.
 Thus $t = 1$. By
 similar arguments, we can get $ p=1$.
Therefore $T=y_0x_1y_{p+1},P=y_0y_1y_{p+1}$ and $|S|=s=|P_1|+1$.

By the maximality of $S$ and Corollary \ref{notSx}, we get that $D$ does not have a $(y_{p+1},y_0)$-ditrail with vertex set $V(P_1)\cup \{x_1\}$
and
\begin{equation} \label{SMd_{S}(x_1)}
d_{S}(x_1)=d_W(x_1)+d_{P_1}(x_1)\leq |P_1|=s-1.
\end{equation}

By the maximality of $S$, $D$ does not have a $(y_{p+1},y_0)$-ditrail with vertex set $V(P_1)\cup \{y_1\}$.
Then by Corollary \ref{notSx}, we get
\begin{equation} \label{SMd_{P_1}(y_1)}
d_{S}(y_1)=d_{P_1}(y_1)\leq |P_1|=s-1.
\end{equation}

By the maximality of $S$,
there is no vertex $z\in R$ satisfying $\{(y_1,z), (z,x_1)\}\subseteq A(D)$
or $\{(x_1,z),(z,y_1)\}\subseteq A(D)$, for any $j\in [p]$. Accordingly,
\begin{equation} \label{SMd_R(x_1)+d_{R}(y_1)}
d_R(x_1)+d_{R}(y_1)\leq 2(n-s-1).
\end{equation}

Combining (\ref{SMd_{S}(x_1)}), (\ref{SMd_{P_1}(y_1)}) and (\ref{SMd_R(x_1)+d_{R}(y_1)}),
note that $\{x_1,y_1\}$ is a pair of dominated nonadjacent vertices,
 we can obtain that
\[ d(x_1) + d(y_1)\leq 2n-4,\]
  contrary to the assumption of Theorem \ref{smd 2n-3}.
  This proves Theorem \ref{smd 2n-3}.

\end{proof}

\begin{Theorem}\label{min n-2}
Let $D$ be
a strong digraph with $n$ vertices.
 If $d(u)+ d(v)\geq 2n -3$ and min$\{d^-(u) + d^+(v), d^+(u) + d^-(v)\} \geq n -2$ for any pair of dominated or dominating nonadjacent vertices $\{u,v\}$ of $D$, then $D$ is supereulerian.
\end{Theorem}

\begin{proof}
Suppose, on the contrary, that $D$ is a nonsupereulerian digraph.
Let $S=y_0y_1\cdots y_py_{p+1}\cdots y_{m-1}y_0$ be a closed ditrail
such that : (1) $|V (S)|=s$ is maximum in $D$,
and (2) subject to (1), $|A(S)|$ is maximum in $D$.
Then $s< n$. Let $R=D-S$.
As $D$ is strong, there exists an $(S, S)$-dipath $T$ with $|T|\geq 3$.

Choose an $(S, S)$-dipath $T$ with $|T|\geq 3$
such that: (1) $|V(P)|$ of the ditrail $P$ is minimum in $S$,
and (2) subject to (1), $|A(P)|$ of the ditrail $P$ is minimum in $S$, where $P$ is a shortest $(u,v)$-ditrail
which
travels along $S$ from $u$ to $v$ such that
the initial vertex $u$ of $P$ is the initial vertex of $T$ and the terminal vertex $v$ of $P$ is the terminal vertex of $T$.
W.l.o.g., write $T=y_0x_1\cdots x_ty_{p+1}$, that is $u=y_0$ and $v=y_{p+1}$.
Let $W = \{y_1,y_2,\cdots, y_p\}$ be the set of internal vertices of $P$,
 $P_1$ be a longest ditrail from $y_{p+1}$ to $y_0$ in $S$.
Then $|P_1|=s-p+c$ and $A(S)=A(P)\cup A(P_1)$, where $c= |W\cap P_1|.$

 By the maximality of $S$, we have $y_0\not=y_{p+1},$ $(y_0,y_{p+1})\notin A(S)$ and $p\geq 1$.
 This together with the fact that $P$ is a
  $(y_{0},y_{p+1})$-ditrail implies $d^{+}_{P}(y_{0})- d^{-}_{P}(y_{0})=1$ and
   $d^{-}_{P}(y_{p+1})- d^{+}_{P}(y_{p+1})=1$.
  If for an integer $k\geq 1$, $d^{+}_{P}(y_{0})=k+1\geq 2$,
  then $d^{-}_{P}(y_{0})=k$. Hence we can denote the ditrail
  $P=y_{0}\cdots y_{0}^{1}\cdots y_{0}^{2} \cdots y_{0}^{k}\cdots y_{p+1}$, where $y_{0}=y_{0}^{1}=y_{0}^{2}=\cdots=y_{0}^{k}$. For any $h\in [k]$, $P_{[y_{0}^{h},y_{p+1}]}=y_{0}^{h}\cdots y_{p+1}$
  is also a ($y_{0},y_{p+1}$)-ditrail which travels along $S$ from $y_{0}$ to $y_{p+1}$.
  If $|V(P_{[y_{0}^{h},y_{p+1}]})|=|V(P)|$, then $|A(P_{[y_{0}^{h},y_{p+1}]})|<|A(P)|$,
  contrary to the choice of $P$. If $|V(P_{[y_{0}^{h},y_{p+1}]})|\not =|V(P)|$,
  then $|V(P_{[y_{0}^{h},y_{p+1}]})|<|V(P)|$ and $|A(P_{[y_{0}^{h},y_{p+1}]})|<|A(P)|$,
  contrary to the choice of $P$. Hence $k=0$, $d^{+}_{P}(y_{0})=1$ and $d^{-}_{P}(y_{0})=0$.
  By similar arguments, we can get that $d^{-}_{P}(y_{p+1})=1$ and $d^{+}_{P}(y_{p+1})=0$. Therefore,
\begin{equation} \label{d^{+}_{P}(y_0)=d^{-}_{P}(y_{p+1})=1}
\begin{split}
d^{+}_{P}(y_0)=d^{-}_{P}(y_{p+1})=1\mbox{ and } d^{-}_{P}(y_0)=d^{+}_{P}(y_{p+1})=0.
\end{split}
\end{equation}

By the choice of $T$ and the maximality of $S$, for any $i\in [t]$, we have $d_W(x_i)=0$.
In particular, $x_i$ and $y_j$ are nonadjacent, for any $i\in [t]$, $j\in [p]$.

By the maximality of $S$, we get that $D$ does not have a $(y_{p+1},y_0)$-ditrail with vertex set $V(P_1)\cup \{x_i\}$. By $d_W(x_i)=0$
and Corollary \ref{notSx}, we can deduce that
\begin{equation} \label{d_{S}(x_1)}
d_{S}(x_i)=d_{W}(x_i)+d_{P_1-W}(x_i)=d_{P_1-W}(x_i)\leq |P_1|-c=s-p.
\end{equation}

If there exists a $(y_{p+1},y_0)$-ditrail $S'$ with vertex set $V(P_1)\cup V(T_{[x_1,x_t]})$,
then $S'\cup P$ is a closed ditrail in $D$ and $|V(S'\cup P)|>|V(S)|$, contrary to the maximality of $S$.
Thus $D$ does not have a $(y_{p+1},y_0)$-ditrail with vertex set $V(P_1)\cup V(T_{[x_1,x_t]})$.
Then by Lemma \ref{notST}, we get
\begin{equation}\label{d^+_{P_1}(x_t)+d^-_{P_1}(x_1)}
d^-_{P_1}(x_1)+d^+_{P_1}(x_t)\leq |V(P_1)|.
\end{equation}

By (\ref{d_{S}(x_1)}) and (\ref{d^+_{P_1}(x_t)+d^-_{P_1}(x_1)}),
\begin{equation}\label{d^+_{P_1}(x_1)+d^-_{P_1}(x_t)}
d^-_{P_1}(x_t)+d^+_{P_1}(x_1)\leq |V(P_1)|-2c=s-p-c.
\end{equation}

By the maximality of $S$,
there is no vertex $z\in R$ satisfying $\{(y_p,z), (z,x_t)\}\subseteq A(D)$
or $\{(x_1,z),(z,y_1)\}\subseteq A(D)$, for any $j\in [p]$. Accordingly,
\begin{equation} \label{d_R(x_1)+d_{R}^-(y_1)+d_{R}^+(y_p)}
d_{R}^+(y_p)+d_{R}^-(x_t)+d^+_R(x_1)+d_{R}^-(y_1)\leq 2(n-s-1).
\end{equation}

By the maximality of $S$, $D$ does not have a $(y_{p+1},y_0)$-ditrail with vertex set $V(P_1)\cup V(P_{[y_1,y_p]})$.
Then by Lemma \ref{notST}, we get
\begin{equation} \label{d_{P_1}^-(y_1)+d_{P_1}^+(y_p)}
d_{P_1}^-(y_1)+d_{P_1}^+(y_p)\leq |P_1|=s-p+c.
\end{equation}

It is obvious that
\begin{equation} \label{d_{W}^-(y_1)+d_{W}^+(y_p)}
d_{W-P_1}^-(y_1)+d_{W-P_1}^+(y_p)\leq2(p-1-c).
\end{equation}

If one of the four inequalities (\ref{d^+_{P_1}(x_1)+d^-_{P_1}(x_t)})-(\ref{d_{W}^-(y_1)+d_{W}^+(y_p)}) is strict, note that $\{x_1,y_1\}$ is a pair of dominated nonadjacent vertices and $\{x_t,y_p\}$ is a pair of dominating nonadjacent vertices, then we can get that
\[n-2+n-2\leq d^-(x_t)+d^+(y_p)+d^-(y_1)+d^+(x_1)\]\[< (s-p-c)+2(n-s-1)+(s-p+c)+2(p-1-c)\leq 2n-4-2c,\]
a contradiction.

Thus, we have $c=0$, in other words, $W\cap P_1=\emptyset$, $|S|=|P|+|P_1|$ and
 the following,
\begin{equation} \label{d_{S}(x_1)=}
d^-_{P_1}(x_t)+d^+_{P_1}(x_1)= |P_1|=s-p.
\end{equation}
\begin{equation} \label{d_R(x_1)+d_{R}^-(y_1)+d_{R}^+(y_p)=}
d_{R}^+(y_p)+d_{R}^-(x_t)+d^+_R(x_1)+d_{R}^-(y_1)= 2(n-s-1).
\end{equation}
\begin{equation} \label{d_{P_1}^-(y_1)+d_{P_1}^+(y_p)=}
d_{P_1}^-(y_1)+d_{P_1}^+(y_p)= |P_1|=s-p.
\end{equation}
\begin{equation} \label{d_{W}^-(y_1)+d_{W}^+(y_p)=}
d_{W}^-(y_1)+d_{W}^+(y_p)= 2(p-1).
\end{equation}

Obviously,
$d_{W}(y_j)\leq 2(p-1)$.
Furthermore,
by the choice of $T$ and the maximality of $S$,
there is no vertex $z\in R$ satisfying $\{(y_j,z), (z,x_i)\}\subseteq A(D)$
or $\{(x_i,z),(z,y_j)\}\subseteq A(D)$, for any $i\in [t],j\in [p]$. Thus,
$d_{R}(x_i)+d_{R}(y_j)\leq 2(n-s-1).$
Combining the two inequalities with (\ref{d_{S}(x_1)}) and
the assumption of the theorem,
note that the pair of nonadjacent vertices $\{x_1, y_1\}$ is dominated by $y_0$ and the pair of nonadjacent vertices $\{x_t, y_p\}$ dominates $y_{p+1}$, we get
\[2n-3\leq d(x_1)+d(y_1)\leq d_{P_1}(y_1)+2n-|P_1|-4\]
and
\[2n-3\leq d(x_t)+d(y_p)\leq d_{P_1}(y_p)+2n-|P_1|-4\]
Accordingly,
\begin{equation} \label{d_{P_1}(y_1)andd_{P_1}(y_p)}
d_{P_1}(y_1)\geq |P_1|+1
\mbox{ and }
d_{P_1}(y_p)\geq |P_1|+1.
\end{equation}

By (\ref{d_{P_1}^-(y_1)+d_{P_1}^+(y_p)=}), if $y_1=y_p$, then $d_{P_1}(y_1)= |P_1|$, contrary to (\ref{d_{P_1}(y_1)andd_{P_1}(y_p)}). Thus $y_1\not=y_p$.
%
By (\ref{d_{P_1}(y_1)andd_{P_1}(y_p)}), there must exist vertices $y_a,y_c\in V(P_1)$ such that $\{(y_a,y_{1}),(y_{1},y_a),(y_{p},y_c),(y_c,y_{p})\}\subseteq A(D)$. Since $W\cap P_1=\emptyset$ and $S=P+P_1$, we have $y_1,y_p\in W,y_1,y_p\not\in V(P_1),y_a,y_c\in V(P_1)$ and $y_a,y_c\not\in W$.
Then $(y_a,y_{1}),(y_{1},y_a),(y_{p},y_c),(y_c,y_{p})\not\in A(P_1)$.

By (\ref{d^{+}_{P}(y_0)=d^{-}_{P}(y_{p+1})=1}),
  we have $d^{-}_{P}(y_{p+1})=|\{(y_p,y_{p+1})\}|=1$ and $d^{+}_{P}(y_{p+1})=0$.
If $y_a=y_{p+1}$, note that $y_p\not=y_1$, then $(y_{p+1},y_{1}),(y_{1},y_{p+1})\not\in A(P)$.
Therefore $(y_{a},y_{1}),(y_{1},y_{a})\not\in A(S)$.
 But then we can get a closed ditrail $S'=S\cup \{(y_a,y_{1}),(y_{1},y_a)\}$ with $|A(S')|>|A(S)|$, contrary to the maximality of $S$. Thus $y_a\not=y_{p+1}$.
 Similarly, we can get that $y_c\not=y_{0}$.

If $y_a\not=y_{0}$, then $(y_{1},y_a),(y_a,y_{1})\not\in A(P)$.
 Therefore $(y_{a},y_{1}),(y_{1},y_{a})\not\in A(S)$.
 But then we can get a closed ditrail $S'=S\cup \{(y_a,y_{1}),(y_{1},y_a)\}$ with $|A(S')|>|A(S)|$, contrary to the maximality of $S$. Thus $y_a=y_{0}$.
 Similarly, we can get that $y_c=y_{p+1}$.
 Then $\{(y_0,y_{1}),(y_{1},y_0),(y_{p},y_{p+1}),(y_{p+1},y_{p})\}\subseteq A(D)-A(P_1)$.

 By the maximality of $S$ with (\ref{d_{P_1}(y_1)andd_{P_1}(y_p)}), for any $y_i\in V(P_1)-y_{0}$ and $y_j\in V(P_1)-y_{p+1}$, we have $|\{(y_i,y_{1}),(y_{1},y_i)\}|= 1$ and $|\{(y_j,y_{p+1}),(y_{p+1},y_j)\}|= 1$.
 By (\ref{d^{+}_{P}(y_0)=d^{-}_{P}(y_{p+1})=1}), $(y_1,y_0),(y_{p+1},y_p)\not\in A(P)$. Then $(y_1,y_0),(y_{p+1},y_p)\not\in A(S)$.
 If $(y_{m-1},y_1)\in A(D)$, note that $(y_{m-1},y_1)\not\in A(S)$,
 then we can get a closed ditrail $S'=S +(y_{m-1},y_1)+(y_{1},y_0)-(y_{m-1},y_{0})$ with $|A(S')|>|A(S)|$, contrary to the maximality of $S$.
 Thus $(y_{m-1},y_1)\not\in A(D)$ and $(y_1,y_{m-1})\in A(D)$.
 Continuing this process, we finally conclude that for any $y_i\in V(P_1)-y_{0}$,
 $(y_i,y_{1})\not\in A(D)$ and $(y_1,y_{i})\in A(D)$.
 Similarly, we can get that for any $y_j\in V(P_1)-y_{p+1}$,
 $(y_p,y_{j})\not\in A(D)$ and $(y_j,y_{p})\in A(D)$.
 In particular, $(y_{p+1},y_{1})\not\in A(D)$.

 Now we have $d^+_{P_1}(y_1)=|P_1|= d^-_{P_1}(y_p)$ and $d^-_{P_1}(y_1)=1= d^+_{P_1}(y_p)$.
 Then by (\ref{d_{P_1}^-(y_1)+d_{P_1}^+(y_p)=}), we obtain $|P_1|=|\{y_{p+1},y_0\}|=2$.
 By the maximality of $S$, $(x_1,y_0),(y_{p+1},x_1)\not \in A(D)$.
 This together with $d_W(x_1)=0$ implies that $d_{S}(x_1)=2$.
  Combining this with the fact that $d_{R}(x_1)+d_{R}(y_j)\leq 2(n-s-1)$
and the assumption of the theorem,
note that the pair of nonadjacent vertices $\{x_1, y_1\}$ is dominated by $y_0$,
we obtain $d_S(y_1)\geq 2s-3$.
Since $d_S(y_1)\leq 2(s-1)$ and $(y_{p+1},y_{1})\not\in A(D)$, $d_S(y_1)= 2s-3$.
That is, for any $y_j\in V(S)-y_{p+1}$, $(y_j,y_1),(y_{1},y_j) \in A(D)$.
 But then we can get a closed ditrail $S'=T\cup \{(y_{j},y_1), (y_{1},y_{j})\}+(y_{p+1},y_0)$, for every $y_j\in V(S)-y_{p+1}$, with $|S'|>|S|$, contrary to the maximality of $S$.
This completes the proof for Theorem \ref{min n-2}.

\end{proof}

\begin{Theorem}\label{5/2-11/2}
Let $D$ be
a strong digraph with $n\geq 2$ vertices.
If $d(u)+ d(v)\geq \frac{5}{2}n -\frac{11}{2}$
for any pair of dominated or dominating nonadjacent vertices $\{u,v\}$ of $D$, then $D$ is supereulerian.
\end{Theorem}

\begin{proof}
If $n\in \{ 2,3\}$, then $D$ is supereulerian as $D$ is strong.
Thus assume that $n\geq 4.$
Assume by contradiction that $D$ is a nonsupereulerian digraph.
Let $S=y_0y_1\cdots y_py_{p+1}\cdots y_{m-1}$\\$y_0$ be a closed ditrail with $|V (S)|=s$ maximized in $D$.
Then $s< n$.
Since $D$ is strong, there exists an $(S, S)$-dipath $T$ with $|T|\geq 3$.
Choose an $(S, S)$-dipath $T$ with $|T|\geq 3$
such that the length of the ditrail $P$ is minimum in $S$, where $P$ is a shortest $(u,v)$-ditrail
which
travels along $S$ from $u$ to $v$ such that
the initial vertex $u$ of $P$ is the initial vertex of $T$ and the terminal vertex $v$ of $P$ is the terminal vertex of $T$.
W.l.o.g., write $T=y_0x_1x_2\cdots x_ty_{p+1}$, that is $u=y_0$ and $v=y_{p+1}$.
Let $W = \{y_1,y_2,\cdots, y_p\}$ be the set of internal vertices of $P$,
$P_1$ be a longest ditrail which travels along $S$ from $y_{p+1}$ to $y_0$ and $R=D-S$.
Then $|P_1|=s-p+c$, where $c= |W\cap P_1|.$
By the maximality of $S$, we have $y_0\not=y_{p+1},$ $(y_0,y_{p+1})\notin A(S)$ and $p\geq 1$.

By the choice of $T$ and the maximality of $S$, for any $i\in [t]$, we have $d_W(x_i)=0$.
In particular, $x_i$ and $y_j$ are nonadjacent, for any $i\in [t]$, $j\in [p]$.

By the maximality of $S$, we get that $D$ does not have a $(y_{p+1},y_0)$-ditrail with vertex set $V(P_1)\cup \{x_i\}$.
Combining $d_W(x_i)=0$
and Corollary \ref{notSx}, we can deduce that
\[
d_{S}(x_i)=d_{W-P_1}(x_i)+d_{P_1}(x_i)=d_{P_1-W}(x_i)\leq |P_1|-c=s-p.
\]

Obviously, $d_{W-P_1}(y_j)\leq 2(p-1-c)$. Then
\[
d_S(y_j)=d_{P_1}(y_j)+d_{W-P_1}(y_j)\leq 2(p-1-c)+d_{P_1}(y_j).
\]

Furthermore,
by the choice of $T$ and the maximality of $S$,
there is no vertex $z\in R$ satisfying $\{(y_j,z), (z,x_i)\}\subseteq A(D)$
or $\{(x_i,z),(z,y_j)\}\subseteq A(D)$, for $i\in [t],j\in [p]$. Thus,
\[
d_{R}(x_i)+d_{R}(y_j)\leq 2(n-s-1).
\]

Combining the three inequalities with
the assumption of the theorem,
note that the pair of nonadjacent vertices $\{x_1, y_1\}$ is dominated by $y_0$ and the pair of nonadjacent vertices $\{x_t, y_p\}$ dominates $y_{p+1}$, we get
\begin{equation} \label{{5}{2}n -{11}{2}1}
\frac{5}{2}n -\frac{11}{2}\leq d(x_1)+d(y_1)\leq d_{P_1}(y_1)+2n-|P_1|-4-c
\end{equation}
and
\begin{equation} \label{{5}{2}n -{11}{2}2}
\frac{5}{2}n -\frac{11}{2}\leq d(x_t)+d(y_p)\leq d_{P_1}(y_p)+2n-|P_1|-4-c.
\end{equation}
Therefore, by $n\geq 4$, we have
\begin{equation} \label{d_{P_1}(y_1)d_{P_1}(y_p)}
\begin{split}
d_{P_1}(y_1)\geq |P_1|+\frac{n-3}{2}+c\geq |P_1|+\frac{1}{2}+c~~~
\\
\mbox{and }
d_{P_1}(y_p)\geq |P_1|+\frac{n-3}{2}+c\geq |P_1|+\frac{1}{2}+c.
\end{split}
\end{equation}

By the maximality of $S$, $D$ does not have a $(y_{p+1},y_0)$-ditrail with vertex set $V(P_1)\cup V(P_{[y_1,y_p]})$.
Then by Lemma \ref{notST}, we get
$
d_{P_1}^-(y_1)+d_{P_1}^+(y_p)\leq |P_1|.
$
If $y_1=y_p$, then $d_{P_1}(y_1)\leq |P_1|$,
contrary to (\ref{d_{P_1}(y_1)d_{P_1}(y_p)}).
Thus $y_1\not=y_p$.
Now we consider two cases in the following.
\\
{\bf Case 1.} $|W|=2$.

In this case, $W=\{y_1,y_p\}$.
By (\ref{d_{P_1}(y_1)d_{P_1}(y_p)}), there must exist vertices $y_a,y_c\in V(P_1)$ such that $\{(y_a,y_{1}),(y_{1},y_a),(y_{p},y_c),(y_c,y_{p})\}\subseteq A(D)$.

If $y_1,y_p\notin V(P_1)$, then $S'=P_1\cup T\cup \{(y_a,y_{1}),(y_{1},y_a),(y_{p},y_c),(y_c,y_{p})\}$ is a
closed ditrail in $D$ and $|S'|>|S|$, contrary to the maximality of $S$.
If $y_1,y_p\in V(P_1)$, then $S'=P_1\cup T$ is a
closed ditrail in $D$ and $|S'|>|S|$, contrary to the maximality of $S$.
Thus $\{y_1,y_p\}\cap V(P_1)=1$.
Assume, w.l.o.g., that $y_1\in V(P_1)$ and $y_p\notin V(P_1)$.
Then $S'=P_1\cup T\cup \{(y_{p},y_c),(y_c,y_{p})\}$ is a
closed ditrail in $D$ and $|S'|>|S|$, contrary to the maximality of $S$.
\\
{\bf Case 2.} $|W|\geq 3$.

In this case, $n\geq s+1= |P_1|+|W|-c+1\geq |P_1|+4-c$.
Then by (\ref{{5}{2}n -{11}{2}1}) and (\ref{{5}{2}n -{11}{2}2}),
we have
\[
d_{P_1}(y_1)\geq |P_1|+\frac{n-3}{2}+c\geq |P_1|+\frac{|P_1|+1-c}{2}+c=\frac{3|P_1|+1+c}{2}\]
and
\[d_{P_1}(y_p)\geq |P_1|+\frac{n-3}{2}+c\geq |P_1|+\frac{|P_1|+1-c}{2}+c=\frac{3|P_1|+1+c}{2}
\]
Accordingly,
\begin{equation} \label{5/2d_{P_1}(y_1)}
\begin{split}
d_{P_1}(y_1)+
d_{P_1}(y_p)\geq 3|P_1|+1+c.
\end{split}
\end{equation}

Combining (\ref{5/2d_{P_1}(y_1)}) with $d^+_{P_1}(y_1)+d^-_{P_1}(y_p)\leq 2|P_1|$, we have $d^-_{P_1}(y_1)+d^+_{P_1}(y_p)\geq |P_1|+1+c$.
Then by Lemma \ref{notST}, $D$
have a $(y_{p+1},y_0)$-ditrail $S'$ with vertex set $V(P_1)\cup W$.
But then we have
 a closed ditrail $S'\cup T$ with $|S'\cup T|>|S|$, contrary to the maximality of $S$.
 The proof of Theorem \ref{5/2-11/2} is complete.

\end{proof}

To show that our results are best possible in some sense, we present Example \ref{fanli 3ge} below. In this example,
we construct an example of a nonsupereulerian digraph with the condition
 $d(u)=d(v)=d^+(u)+d^-(v)=d^+(v)+d^-(u)= n-2$ for a pair of dominated nonadjacent vertices $\{u,v\}$.

\begin{Example} \label{fanli 3ge}
We construct a strong
digraph $D$ with $|V(D)|=n = |\{u,v\}\cup V(K^*_{n_1})\cup V(K^*_{n_2})|=n_1+n_2+2$ and the arcs of $D$ are shown in (i) and (ii) below.
(See Figure 1).
\\
(i) $K^*_{n_1}$ and $K^*_{n_2}$ are complete digraph.
\\
(ii) $(w',w)\in A(D)$, $N^{+}(V(K^*_{n_2})) =\{u,v\}\cup V(K^*_{n_1})$,
$N^{+}(u)=N^{+}(v)=V(K^*_{n_1})$ and $N^{-}(u)=N^{-}(v)=V(K^*_{n_2})$, where $w'\in V(K^*_{n_1})$ and $w\in V(K^*_{n_2})$.
\end{Example}

\[\begin{tikzpicture}
[x=0.8cm, y=0.3cm, every edge/.style={draw, postaction={decorate,decoration={markings,mark=at position 0.6 with {\arrow{>}}}}}]
\draw[fill=black] (-0.2,0) circle (0.08cm);
\draw[fill=black] (3.2,0) circle (0.08cm);
\draw[fill=black] (-0.8,-8) circle (0.08cm);
\draw[fill=black] (3.5,-8) circle (0.08cm);
\draw[] (-1.5,0) ellipse (1.5 and 4.5);
\draw[] (4.5,0) ellipse (1.5 and 4.5);
\node at (-0.6,0){$w'$};
\node at (3.6,0){$w$};
\node at (-0.8,-8.8){$u$};
\node at (3.5,-8.8){$v$};
\node at (-2,0){$K^*_{n_1}$};
\node at (5,0){$K^*_{n_2}$};
\path
(-0.2,0)edge[] (3.2,0)
;
\draw[thick](-0.8,-8)--(-1.5,-3);
\draw[directed](-0.8,-8)--(-1.5,-3);
\draw[thick](3.5,-8)--(-1,-3);
\draw[directed](3.5,-8)--(-1,-3);
\draw[thick](4,-3)--(-0.8,-8);
\draw[directed](4,-3)--(-0.8,-8);
\draw[thick](4.5,-3)--(3.5,-8);
\draw[directed](4.5,-3)--(3.5,-8);
 \path
       (4,2) edge[bend right=60] (-1,2);
       \draw[thick](4,2) edge[bend right=60] (-1,2);
       \node at (1.4,-10.5){Figure 1.~  The strong
digraph $D$.};
\end{tikzpicture}\]

It is not difficult to show that the digraph $D$ of Figure 1 is nonsupereulerian.
In fact, as $N^{-}_D(u)=N^{-}_D(v)=V(K^*_{n_2})$, any spanning eulerian subdigraph (if it exists) $S$ of $D$ has to contain at least one arc in $(V(K^*_{n_2}),u)_D$ and one arc in $(V(K^*_{n_2}),v)_D$,
that is $|N^{-}_S(V(K^*_{n_2}))|=|N^{+}_S(V(K^*_{n_2}))|\geq 2$ in any spanning eulerian subdigraph (if it exists) $S$ of $D$.
However $|N^{-}_D(V(K^*_{n_2}))|= 1$ in $D$,
so such a spanning eulerian subdigraph does not exist.

In particular, the nonsupereulerian digraph $D$ of Figure 1 is semicomplete multipartite with the condition
 $d(u)+d(v)= 2n-4$ for a pair of dominated nonadjacent vertices $\{u,v\}$.
 Thus the condition of Theorems \ref{smd 2n-3} is sharp.

Moreover, it is obvious that $d(u)+d(v)= 2n-4$ and min$\{d^-(u) + d^+(v), d^+(u) + d^-(v)\} = n -2$ and $\{u,v\}$ is the only pair of dominated and dominating nonadjacent vertices the digraph $D$ of Figure 1.
Therefore, 
Example \ref{fanli 3ge} demonstrates that there are infinitely many nonsupereulerian digraphs satisfying
$d(u)+d(v)\geq 2n-4$ and min$\{d^-(u) + d^+(v), d^+(u) + d^-(v)\} \geq n -2$ for a pair of dominated and dominating nonadjacent vertices $\{u,v\}$.
Thus conditions of Theorems \ref{min n-2} cannot be weakened by more than a constant.
So the condition of Theorems \ref{min n-2} is sharp.

 Finally, if $n_1=n_2=1$ or $n_1=1,n_2=2$ or $n_1=2,n_2=1$, then the nonsupereulerian digraph $D$ of Figure 1 satisfies the condition
 $d(u)+d(v)\geq \frac{5}{2}n -\frac{13}{2}$ for a pair of dominated and dominating nonadjacent vertices $\{u,v\}$. Therefore
the condition of Theorems \ref{5/2-11/2} is sharp in some sense.
\\
\\
\noindent {\bf Data availability statement}
 No data, models, or code were generated or used during the study.\\
 \\
\noindent {\bf Conflict of interest}
The authors have not disclosed any competing interests.

\end{document}